\documentclass[reqno, 11pt]{amsart} % draft
\usepackage{amsfonts, amsmath, amssymb, amsthm}
\usepackage[margin=2.75cm, heightrounded]{geometry}
\usepackage{fancybox}
\usepackage{hhline, float}
\usepackage{mathrsfs}
\usepackage{nccmath}
\usepackage[dvipsnames]{xcolor}
\usepackage[colorlinks=true]{hyperref}
\hypersetup{citecolor=red, linkcolor=blue}

\allowdisplaybreaks
\numberwithin{equation}{section}

\usepackage[dvipsnames]{xcolor}
\usepackage[colorlinks=true]{hyperref}
\hypersetup{citecolor=red, linkcolor=blue}

\usepackage{amsthm}

\newtheorem{theorem}{Theorem}[section]

\newtheorem{lemma}[theorem]{Lemma}
\newtheorem{proposition}[theorem]{Proposition}
\newtheorem{corollary}[theorem]{Corollary}

\theoremstyle{definition}
\newtheorem{definition}[theorem]{Definition}
\newtheorem{remark}[theorem]{Remark}

\makeatletter
\@namedef{subjclassname@2020}{%
\textup{2020} Mathematics Subject Classification}
\makeatother

\begin{document}

\title[Entropy instability and rigidity of self-similar profiles]
{Entropy instability and rigidity for the exponential semilinear heat equation}

\author{Yuxia Guo}
\address[Yuxia Guo]{Department of Mathematical Sciences, Tsinghua University}
\email{yguo@tsinghua.edu.cn}

\author{Jionghao Lv}
\address[Jionghao Lv]{Department of Mathematical Sciences, Tsinghua University}
\email{lvjh26@mails.tsinghua.edu.cn}

\begin{abstract}
In this paper, we study bounded self-similar solutions for the exponential semilinear heat equation via the $F$-functional and entropy. Motivated by recent developments in mean curvature flow, we prove that every non-constant bounded solution of the self-similar equation
is entropy unstable. We also show that the trivial profile $w\equiv0$ is quantitatively isolated among bounded solutions. These results provide both a variational instability theorem for nontrivial profiles and a rigidity theorem for the trivial profile.
\end{abstract}
\maketitle


\section{Introduction}

We consider  the following Cauchy problem for the exponential semilinear heat equation:

\begin{equation}\label{initial problem}
\left\{
\begin{array}{l}
\partial_t u = \Delta u + e^{u},
\quad \text{in } \mathbb{R}^n \times (0,T), \\[0.3em]
u(\cdot,0) = u_0\in L^{\infty}(\mathbb{R}^{n}).
\end{array}
\right.
\end{equation}
\eqref{initial problem} is a basic reaction-diffusion model with a rapidly growing source term. In combustion theory, the evolution model \eqref{initial problem} was proposed by Frank-Kamenetskii \cite{FK-1969} and is known as the solid-fuel ignition model \cite{B-E-1989}. Related reaction-diffusion models also arise in localized chemical reactions due to the presence of catalysts \cite{Chadam-Y-1994}.

From the view point of mathmatics, an important topic in the study of  \eqref{initial problem} is the blow-up behavior, which has been extensively studied recently. One of the first blow-up results for model \eqref{initial problem} was obtained by Friedman and McLeod \cite{Friedman-M-1985}, in which the authors  established upper and lower bounds for the blow-up rate under some  conditions on the initial data, see also \cite{B-E-1989} for a similar estimate. Liu \cite{Liu-1989} derived the precise type I blow-up rate and proved the nondegeneracy of blow-up points. Fila and Pulkkinen \cite{Fila-P-2008} constructed a non-constant self-similar blow-up profile for \eqref{initial problem}, while Pulkkinen \cite{Pulkkinen-2011} further analyzed the corresponding final time blow-up profiles. Recently, Ji et al. \cite{Ji-L-C-2019} ruled out type II blow-up for radial solutions to \eqref{initial problem} in a range of dimensions. Souplet \cite{Souplet-2022} obtained the sharp upper estimates for the final profile and for the refined space-time behavior while Chabi \cite{Chabi-2025} extended these results to more general non-scale-invariant exponential-type nonlinearities. For further studies on the exponential semilinear heat equation, we refer to \cite{Fila-M-P-2005, Fujishima-2018, Ghoul-N-Z-2017, Ghoul-N-Z-2018, Herrero-V-1993, Souplet-T-2016, Tello-2006, V-1999} and the references therein.

A solution $u$ of \eqref{initial problem} is said to blow up in finite time at $T<\infty$ if
\[
\limsup_{t\rightarrow T} \|u(\cdot,t)\|_{L^\infty(\mathbb{R}^n)} = +\infty.
\]
According to  the growth rate of the solution near the singular time, finite-time blow-up can be divided into two distinct types: namely,  type I and type II. More precisely, the blow-up is said to be of type I if
\[
\limsup_{t\rightarrow T}
\bigl| \log(T-t) + \|u(\cdot,t)\|_{L^\infty(\mathbb{R}^n)} \bigr|
< +\infty,
\]
and of type II if
\[
\limsup_{t\rightarrow T}
\bigl| \log(T-t) + \|u(\cdot,t)\|_{L^\infty(\mathbb{R}^n)} \bigr|
= +\infty.
\]
In the type I regime, we consider the corresponding self-similar transform
\begin{equation*}
\widetilde w(y,\tau)=u(x,t)+\log(T-t),\quad x=\sqrt{T-t}\,y,\qquad T-t=e^{-\tau}.
\end{equation*}
If $u$ satisfies \eqref{initial problem}, then $\widetilde{w}$ satisfies
\begin{equation}\label{self-similar equation}
\partial_\tau \widetilde w-\Delta \widetilde w+\frac12 y\cdot\nabla \widetilde  w-e^{\widetilde  w}+1=0.
\end{equation}
In particular, a stationary solution $w$ of  \eqref{self-similar equation} satisfies
\begin{equation}\label{stationary problem}
\Delta w-\frac12 y\cdot\nabla w+e^{w}-1=0,\quad \text{in }\mathbb{R}^n.
\end{equation}
The trivial solution $w\equiv 0$ corresponds to the spatially homogeneous type I profile
$$u(x,t)= -\log(T-t).$$

One of the aims in this paper is to determine which bounded solutions of the self-similar equation \eqref{stationary problem} can serve as stable type I blow-up profiles. To study this problem, we follow the idea from Colding-Ilmanen-Minicozzi \cite{Colding-I-M2015}, Colding-Minicozzi  \cite{Colding-M-2012} and  Huisken \cite{Huisken-1990}, where the analysis of self-similar singularities in mean curvature flow  has been developed through Huisken's monotonicity formula \cite{Huisken-1990}, the entropy and stability theory of Colding-Minicozzi \cite{Colding-M-2012}, and the rigidity theorem of Colding-Ilmanen-Minicozzi \cite{Colding-I-M2015}.

Motivated by these developments, analogous variational and entropy-based methods have recently been applied to semilinear heat equations. For the supercritical Fujita equation, Wang-Wei-Wu \cite{Wang-W-W2025 F stability} introduced the $F$-functional, entropy, $F$-stability and entropy stability, and they proved that constant solutions minimize the entropy among bounded positive self-similar solutions. Moreover, they established an entropy gap between constant and non-constant solutions, with an application to the stratification and rectifiability of the type I blow-up set. In their subsequent work, they proved a Liouville theorem for ancient solutions which are close to the ODE profile at large scales, which implies a stability property for ODE blow-ups in the supercritical Fujita equation \cite{Wang-W-W2025 Liouville}. Recently, Ao et al. \cite{Ao-G-Y-Y2025} adapted this $F$-functional and entropy approach to the exponential semilinear heat equation \eqref{initial problem} and established the corresponding variational structure for bounded solutions of \eqref{stationary problem}. However, to our knowledge, the entropy stability of non-constant bounded solutions to \eqref{stationary problem} was still open in the exponential case.

In this paper we use  the Gaussian variational structure associated with \eqref{stationary problem} to study the entropy instability and rigidity  for the exponential semilinear heat equation.

Denote by $\lambda(w)$ the entropy associated with the profile $w$. Roughly speaking, $\lambda(w)$ measures the maximal value of the $F$-functional energy of the profile under all possible translations and dilations. We postpone the  precise definitions to Section 2.

Our first result shows that any bounded non-constant solution of \eqref{stationary problem} satisfying a mild decay assumption \eqref{eq:w-decay1} is entropy unstable.
\begin{theorem}\label{theorem 1}
    Assume $w$ is a bounded non-constant solution of \eqref{stationary problem}. If there exists a  positive constant $C$ such that
\begin{equation}\label{eq:w-decay1}
|w(y)|\le C(1+|y|)^{-1}\quad \text{in }\mathbb{R}^n,
\end{equation}
then there exists a variation $w_{s}$ with $w_{0}=w$ such that
\[
\lambda(w_{s})<\lambda(w)
\]
for all $s\neq0$. In particular, $w$ is not entropy stable.
\end{theorem}

Our second  result is a rigidity theorem for the trivial profile $w\equiv0$. We begin by introducing the following notation. For $R>1$, if $w_{1},w_{2}$ are two continuous functions on $\overline{B_{R}(0)}$, set
\begin{equation*}
\mathrm{dist}_R(w_1,w_2)=\sup_{y\in B_R(0)}|w_1(y)-w_2(y)|.
\end{equation*}

Using this metric, the following rigidity theorem asserts that the constant solution $0$ is strictly isolated in the space of bounded solutions.
\begin{theorem}\label{theorem 2}
    Let $m>0$. There exist $\widetilde{\varepsilon}$ and $\widetilde{R}$ depending only on $n,m$ such that if $w$ is a smooth solution of \eqref{stationary problem} in $\mathbb{R}^{n}$ satisfying
 $ \|w\|_{L^{\infty}(\mathbb{R}^{n})}\leq m,$ in $\overline{B_{\widetilde{R}}(0)}$ and
\begin{equation}\label{eq:rigidity}
\mathrm{dist}_{\widetilde{R}}(w,0)\leq\widetilde{\varepsilon},
\end{equation}
then $w\equiv0$.
\end{theorem}

Theorems \ref{theorem 1} and \ref{theorem 2} complement each other. The first theorem excludes entropy stability for non-constant bounded profiles, while the second result shows that the trivial profile $w\equiv0$ is strictly isolated in the class of bounded solutions. In this sense, this paper provides both a variational instability result and a Liouville-type rigidity result for the self-similar equation \eqref{stationary problem}.

 Our main strategies to prove the results are as following.

The proof of the entropy instability (Theorem \ref{theorem 1}) mainly relies on constructing a suitable perturbation of the profile whose entropy is strictly smaller. A key step is to derive the  decay estimates for $\nabla w$, the first eigenfunction $f$ of the linearized operator associated with \eqref{stationary problem}, and its gradient $\nabla f$ at infinity. We obtain these pointwise estimates by rewriting \eqref{stationary problem} through the Ornstein-Uhlenbeck resolvent and analyzing the corresponding kernel representation (Lemma \ref{lem nabla and hessian w decay}), as well as by applying the maximum principle together with suitable barrier arguments (Lemma \ref{lem eigenfunction decay}). Then we perturb the profile $w$ in the direction of $f$. To rigorously show that this perturbation strictly decreases the global entropy, we analyze the multi-parameter functional over all translations and dilations.

For the rigidity result (Theorem \ref{theorem 2}), our proof relies mainly on an iteration scheme built on two complementary estimates. Starting from the assumption that the profile $w$ is sufficiently close to the constant profile $0$ on a large ball, this control becomes stronger on a slightly smaller ball, with the radius reduced by a fixed amount. This follows from elliptic regularity and weighted energy estimates, which force both the profile $w$ and an associated auxiliary quantity $\Lambda(w):=2+y\cdot\nabla w$ to be close to $0$ and $2$ respectively. The second estimate shows that once such an improved estimate is established, it can be propagated to a larger ball whose radius is increased by a fixed factor. By iterating these two estimates,  we can improve this control to arbitrarily large regions and eventually get global control. This global estimate implies that the auxiliary quantity $\Lambda(w)$ has a fixed sign, thereby proving the theorem.

The rest of the paper is organized as follows. Section 2 is the preliminaries, where we introduce some definitions and collect several lemmas that will be used later. In Section 3, we prove the entropy instability result (Theorem \ref{theorem 1}). In Section 4, we establish the rigidity theorem (Theorem \ref{theorem 2}).

Throughout the paper, we use $C$ to denote a positive constant which may vary from line to line.
\section{Preliminaries}

In this section, we introduce the $F$-functional and the entropy for the exponential semilinear heat equation. We also recall several definitions and auxiliary results. Most of these results are taken from \cite{Ao-G-Y-Y2025}, building on earlier works such as \cite{Wang-W-W2025 F stability}.

We begin by introducing the $F$-functional associated with the self-similar equation.
\begin{definition}[$F$-functional]
For any bounded smooth function $w$, the $F$-functional is defined by
\begin{equation*}
\begin{split}
F_{x_0,t_0}(w)
:=& \frac{1}{2}(-t_0)\int_{\mathbb{R}^n} |\nabla w|^2 G(y-x_0,t_0){\rm d}y
- (-t_0)\int_{\mathbb{R}^n} e^{w} G(y-x_0,t_0){\rm d}y  \\
&\quad + \int_{\mathbb{R}^n} w G(y-x_0,t_0){\rm d}y
+ \log(-t_0),
\end{split}
\end{equation*}
where for any $(y,t)\in\mathbb{R}^{n}\times(-\infty,0),$
\[
G(y,t)=(-4\pi t)^{-n/2}e^{\frac{|y|^2}{4t}}.
\]
In particular, we set
\begin{equation}\label{gaussian kernel}
    \rho(y):=G(y,-1)=(4\pi)^{-n/2}e^{-\frac{|y|^{2}}{4}},
\end{equation}
which is the Gaussian kernel. With this choice, we write the weighted energy by
\[
E(w):=F_{0,-1}(w),
\]
namely,
\begin{equation*}
E(w)
= \frac12\int_{\mathbb{R}^n}|\nabla w|^2\rho{\rm d}y
-\int_{\mathbb{R}^n}e^w\rho{\rm d}y
+\int_{\mathbb{R}^n}w\rho{\rm d}y.
\end{equation*}
\end{definition}
We next define the critical point of $F$-functional.
 \begin{definition}
     Let $w$ be a bounded smooth function. If for all variations $x(s)$, $t(s)$ satisfying $x(0)=x_{0}$, $t(0)=t_{0}$ and $\phi\in C_{0}^{\infty}(\mathbb{R}^{n})$, we have
\[
\left.\frac{\partial}{\partial s}
\bigl( F_{x(s),\,t(s)}(w + s\phi) \bigr)\right|_{s=0} = 0 ,
\]
then we say that $w$ is a critical point for the functional $F_{x_{0},t_{0}}$.
 \end{definition}
Since the value of $F$-functional depends on the choice of $(x_{0},t_{0})$, it is natural to consider its supremum over all such parameters. This leads to the definition of entropy.
\begin{definition}[Entropy]
    The entropy $\lambda(w)$ of a bounded smooth function $w$ is defined as
\[\lambda(w):=\sup_{x_0\in\mathbb{R}^n,\ t_0\in(-\infty,0)}F_{x_0,t_0}(w).\]
\end{definition}
Following \cite{Wang-W-W2025 F stability}, we also adopt the notions of $F$-stability and entropy stability.
\begin{definition}[$F$-stable]\label{def:f-stable}
    Let $w$ be a bounded solution of \eqref{stationary problem}. If for any $\phi\in C_{0}^{\infty}(\mathbb{R}^{n})$, there exists variations $x(s)$, $t(s)$ with
\[
x(0)=0,\quad t(0)=-1,\quad x'(0)=y_0,\quad t'(0)=h,\quad
x''(0)=y_0',\quad t''(0)=h',
\]
such that
\[
\left.
\frac{\partial^2}{\partial s^2}
\bigl( F_{x(s),\,t(s)}(w+s\phi) \bigr)
\right|_{s=0} \ge 0,
\]
then we say that $w$ is $F$-stable.
\end{definition}
\begin{definition} [Entropy stable]
    A bounded function $w$ is said to be entropy stable if the entropy functional attains a local minimum at $w$.
\end{definition}

We now introduce several weighted spaces which will be used in the spectral analysis of the linearized operator. Define
\[
L^q_\rho(\mathbb{R}^n)
=
\left\{
g \in L^q_{\mathrm{loc}}(\mathbb{R}^n)
\,:\,
\int_{\mathbb{R}^n} |g(y)|^q \rho {\rm d}y < \infty
\right\}
\]
and
\[
H^1_\rho(\mathbb{R}^n)
=
\left\{
g \in L^2_\rho(\mathbb{R}^n)
\,:\,
|\nabla g| + |g| \in L^2_\rho(\mathbb{R}^n)
\right\},
\]
where $\rho$ is defined in \eqref{gaussian kernel}. The inner product on $L^2_\rho(\mathbb{R}^n)$ is defined by
\begin{equation*}
\langle \psi_1, \psi_2 \rangle_\rho
=
\int_{\mathbb{R}^n} \psi_1 \psi_2 \rho {\rm d}y.
\end{equation*}
Then both $L^2_\rho(\mathbb{R}^n)$ and $H^1_\rho(\mathbb{R}^n)$ are Hilbert spaces.

For any bounded solution of \eqref{stationary problem}, we define the linear operator $\mathcal{L}$ by:
\begin{equation}\label{linear operator}
\mathcal{L}\psi=-\Delta \psi + \frac{y}{2}\cdot \nabla \psi - e^{w}\psi.
\end{equation}
It is easy to verify that $\mathcal{L}$ is a self-adjoint operator in $L^2_\rho(\mathbb{R}^n)$ with domain $H^1_\rho(\mathbb{R}^n)$ and the natural embedding $\iota : H^1_\rho(\mathbb{R}^n) \hookrightarrow L^2_\rho(\mathbb{R}^n)$ is compact. Therefore, by standard spectral theory, we obtain the following result, see Corollary 2.2 in \cite{Ao-G-Y-Y2025}.
\begin{corollary}\label{cor eigenvalue problem}
Suppose that $w$ is a bounded solution of \eqref{stationary problem} and $\mathcal{L}$ is defined in \eqref{linear operator}. Then
\begin{itemize}

\item[\textnormal{(1)}] $\mathcal L$ has a sequence of eigenvalues $\lambda_1 < \lambda_2 \le \cdots$.

\item[\textnormal{(2)}] There is an orthogonal basis $\{f_k\}$ in $L^2_{\rho}(\mathbb R^n)$ with $\mathcal{L}f_{k}=\lambda_{k}f_{k}$.

\item[\textnormal{(3)}] The smallest eigenvalue $\lambda_1$ admits the variational characterization
\begin{equation*}
    \lambda_1
= \inf_{\psi \in H^1_{\rho}(\mathbb{R}^n)}
\frac{\displaystyle \int_{\mathbb{R}^n} |\nabla \psi|^2 \rho{\rm d}y
      - \int_{\mathbb{R}^n} e^{w} \psi^2 \rho {\rm d}y}
     {\displaystyle \int_{\mathbb{R}^n} \psi^2 \rho {\rm d}y}.
\end{equation*}

\item[\textnormal{(4)}] Any eigenfunction associated with $\lambda_1$ does not change sign.
\end{itemize}
\end{corollary}
The following lemma shows two explicit eigenfunctions of the operator $\mathcal{L}$ arising from translations and scaling variations.
\begin{lemma}\label{lem eigenvalue problem}
   Suppose that $w$ is a bounded solution of \eqref{stationary problem}. Then
\begin{equation}\label{eigen -1/2}
    \mathcal{L} w_i = -\frac{1}{2} w_i, \qquad i = 1,2,\ldots,n,
\end{equation}
\begin{equation}\label{eigen-1}
    \mathcal{L}\,(y \cdot \nabla w + 2) = -(y \cdot \nabla w + 2),
\end{equation}
where $w_{i}=\partial_{i}w$.
\end{lemma}
\begin{proof}
    Since $w$ solves \eqref{stationary problem}, for  $i=1,2,\cdots,n$, we derive
\[
\Delta w_i-\frac{y}{2}\cdot\nabla w_i-\frac12 w_i+w_{i}e^{w}=0.
\]
Hence we prove \eqref{eigen -1/2}.

To get \eqref{eigen-1}, we consider the exponential-type transformations:$$w_\lambda(y)=2\log\lambda+w(\lambda y).$$ Since $w$ satisfies \eqref{stationary problem}, $w_{\lambda}$ satisfies
\[
\Delta w_\lambda
-\frac{\lambda^2}{2}y\cdot\nabla w_\lambda
+e^{w_\lambda}
-\lambda^2
=0.
\]
Taking derivative with respect to $\lambda$ at $\lambda=1$, we obtain
\begin{equation*}
0=\Delta(2+y\cdot\nabla w)
-\frac{y}{2}\cdot\nabla(2+y\cdot\nabla w)
+e^{w}(2+y\cdot\nabla w)
-(2+y\cdot\nabla w),
\end{equation*}
which is \eqref{eigen-1}.
\end{proof}

% The following regularity result for $u$ is taken from \cite{Ao-G-Y-Y2025}.
% \begin{lemma}\label{lem regularity u}
% Suppose $u$ is a solution of \eqref{initial problem} and satisfies
% \begin{equation*}
%     |\log(-t) + u(x,t)| \le C_1, \quad \text{in } \mathbb{R}^n \times (-\infty, 0).
% \end{equation*}
% Then
% \begin{equation*}
%     (-t)^{\frac{1}{2}}|\nabla u| + (-t)|\nabla^2 u| + (-t)^{\frac{3}{2}}|\nabla^3 u| \le C, \quad \text{in } \mathbb{R}^n \times (-\infty, 0),
% \end{equation*}
% where $C$ depends only on $n$ and $C_{1}$.
% \end{lemma}

The next lemma characterizes the constant solution of \eqref{stationary problem}.
\begin{lemma}\cite{Ao-G-Y-Y2025}\label{lem not change sign}
Suppose that $w$ is a bounded solution of \eqref{stationary problem}. Then $w$ is the constant solution of\eqref{stationary problem} if and only if the function
\begin{equation*}
\Lambda(w)(y) = 2 + y \cdot \nabla w(y)
\end{equation*}
does not change the sign in $\mathbb{R}^{n}$.
\end{lemma}
The next result shows that non-constant bounded solutions are not $F$-stable.
\begin{lemma}\cite{Ao-G-Y-Y2025}\label{lem F-stable}
   Suppose that $w$ is a bounded solution of \eqref{stationary problem}.
If $w$ is not the constant solution, then $w$ is not $F$-stable. Moreover, for any $\varepsilon>0$ sufficiently small, there exists $\delta>0$ such that
\[
   \sup\Big\{F_{x_0,t_0}(w): |x_0|+|\log(-t_0)|> \varepsilon\Big\}\le \lambda(w)-\delta.
\]
\end{lemma}

The following lemma characterizes critical points of the $F$-functional.
\begin{lemma}\cite{Ao-G-Y-Y2025}\label{lem critical point}
   A bounded smooth function $w$ is a critical point of $F_{x_{0},t_{0}}$ if and only if $w$ solves
\[
\Delta w
+ \frac{y-x_0}{2t_0}\cdot\nabla w
+ e^{w}
+ \frac{1}{t_0}
=0,
\quad \text{in } \mathbb{R}^n.
\]
\end{lemma}

Finally, we state two lemmas about the second variation formula and the entropy attainment.
\begin{lemma}\cite{Ao-G-Y-Y2025}\label{lem second varaition}
 Suppose that $w$ is a bounded solution of \eqref{stationary problem}. Let $x(s)$ and $t(s)$ be variations of
 \[
x(0)=0,\quad t(0)=-1,\quad x'(0)=y_0,\quad t'(0)=h,\quad x''(0)=y_0',\quad t''(0)=h'.
\]
Then for any $\phi\in C_0^\infty(\mathbb{R}^n)$
\begin{equation*}
\begin{aligned}
\left.\frac{\partial^2}{\partial s^2}
\Big(F_{x(s),t(s)}(w+s\phi)\Big)\right|_{s=0}
={}&\;
h\int_{\mathbb{R}^n}(\nabla w\cdot y+2)\phi\rho{\rm d}y
+\int_{\mathbb{R}^n}|\nabla\phi|^2\rho{\rm d}y \\
&-\int_{\mathbb{R}^n}e^{w}\phi^2\rho{\rm d}y
-\int_{\mathbb{R}^n}(\nabla w\cdot y_0)\phi\rho{\rm d}y \\
&-\frac12\int_{\mathbb{R}^n}(\nabla w\cdot y_0)^2\rho{\rm d}y
-h^2\int_{\mathbb{R}^n}\left(1+\frac{y}{2}\cdot\nabla w\right)^2\rho{\rm d}y.
\end{aligned}
\end{equation*}
\end{lemma}
\begin{lemma}\cite{Ao-G-Y-Y2025}\label{lem entropy attainable}
If $w$ is a bounded solution of \eqref{stationary problem}, then $\lambda(w)$ is attained at $(0,-1)$ and satisfies $\lambda(w)=E(w)$.
\end{lemma}

\section{Entropy instability}
To prove Theorem \ref{theorem 1}, we first prove a regularity result for solutions of \eqref{stationary problem} which decay at infinity.

\begin{lemma}\label{lem nabla and hessian w decay}
    Assume $w$ is a bounded solution of \eqref{stationary problem} and satisfies, for some positive constant $C_{1}$,
\begin{equation}\label{eq:w-decay}
|w(y)|\le C_{1}(1+|y|)^{-1}\quad \text{in }\mathbb{R}^n.
\end{equation}
Then there exists a positive $C$ such that
\begin{equation*}
    |\nabla w|\leq C(1 + |y|)^{-1}, \quad \text{in }  \mathbb{R}^n.
\end{equation*}
\end{lemma}
\begin{proof}
    Define $F(y):=e^{w(y)}-1+w(y)$, then  we can rewrite \eqref{stationary problem} as
\begin{equation*}
    w-L_{OU}w=F,\quad \text{in }\mathbb{R}^{n},
\end{equation*}
where $L_{OU}$ is the Ornstein-Uhlenbeck operator $\Delta-\frac{y}{2}\cdot\nabla$. By the resolvent formula for the Ornstein-Uhlenbeck semigroup \cite{Metafune-P-P-2002}, $w$ can be represented as
\begin{equation*}
    w(y)=\int_0^\infty e^{-t}\int_{\mathbb{R}^n} p_t(y,\eta)F(\eta){\rm d}\eta{\rm d}t,
\end{equation*}
where $p_{t}(y,\eta)$ is the Ornstein-Uhlenbeck kernel
\begin{equation*}
p_t(y,\eta)
=
\frac{1}{\bigl(4\pi(1-e^{-t})\bigr)^{n/2}}
\exp\!\left(
-\frac{|\eta-e^{-t/2}y|^2}{4(1-e^{-t})}
\right),
\quad t>0.
\end{equation*}
We first claim that $F$ satisfies the same decay property as $w$. Since $w$ is bounded, say $|w|\leq C_1$, define $g(\xi):=e^{\xi}-1+\xi$, then
\begin{equation*}
    g'(\xi)=e^{\xi}+1,\quad 1+e^{-C_1}\leq g'(\xi)\leq 1+e^{C_1},\quad \xi\in[-C_1,C_1].
\end{equation*}
By the mean value theorem,
\begin{equation*}
    |F(y)|=|g(w(y))-g(0)|\leq(1+e^{C_1})|w(y)|,
\end{equation*}
hence,
\begin{equation}\label{f decay}
    |F(y)|\leq C(1+|y|)^{-1},\quad \text{in }\mathbb{R}^{n}.
\end{equation}
 We next obtain the gradient estimate. Differentiating the kernel with respect to $y$, we derive
\begin{equation*}
\nabla_y p_t(y,\eta)=\frac{e^{-t/2}}{2(1-e^{-t})}\,(\eta-e^{-t/2}y)\,p_t(y,\eta).
\end{equation*}
Therefore
\begin{equation*}
\nabla w(y)
=
\int_0^\infty \frac{e^{-3t/2}}{2(1-e^{-t})}
\int_{\mathbb{R}^n} p_t(y,\eta)(\eta-e^{-t/2}y)F(\eta){\rm d}\eta{\rm d}t.
\end{equation*}
Then we have
\begin{equation}\label{abs nabla w}
|\nabla w(y)|\le
\int_0^\infty \frac{e^{-3t/2}}{2(1-e^{-t})}
\int_{\mathbb{R}^{n}} p_t(y,\eta)|\eta-e^{-t/2}y||F(\eta)|{\rm d}\eta{\rm d}t.
\end{equation}
Notice that
\begin{equation*}
    1+e^{-t/2}|y|\le (1+|\eta|)(1+|\eta-e^{-t/2}y|),
\end{equation*}
together with \eqref{f decay}, we have
\begin{equation}\label{f estimate}
|F(\eta)|
\le
C(1+|\eta|)^{-1}
\le
C\frac{1+|\eta-e^{-t/2}y|}{1+e^{-t/2}|y|}.
\end{equation}
A combination of \eqref{abs nabla w} and \eqref{f estimate} leads to
\begin{equation}\label{nabla w estimate 1}
|\nabla w(y)|
\le
C\int_0^\infty \frac{e^{-3t/2}}{1-e^{-t}}
(1+e^{-t/2}|y|)^{-1}
\int_{\mathbb{R}^{n}} p_t(y,\eta)|\eta-e^{-t/2}y|(1+|\eta-e^{-t/2}y|){\rm d}\eta{\rm d}t.
\end{equation}
Next we compute the inner integral $$I_{in}:=\int_{\mathbb{R}^{n}} p_t(y,\eta)|\eta-e^{-t/2}y|(1+|\eta-e^{-t/2}y|){\rm d}\eta.$$ Indeed, setting
\[r=\eta-e^{-t/2}y,\quad s=1-e^{-t},\]
we obtain
\begin{equation*}
    \int_{\mathbb{R}^n}p_t(y,\eta) |\eta - e^{-t/2}y|(1 + |\eta - e^{-t/2}y|) {\rm d}\eta = \int_{\mathbb{R}^n}\frac{1}{(4\pi s)^{n/2}} e^{-\frac{|r|^2}{4s}} |r|(1 + |r|){\rm d}r.
\end{equation*}
We denote $$J_{m}:=\int_{\mathbb{R}^n}\frac{1}{(4\pi s)^{n/2}} e^{-\frac{|r|^2}{4s}} |r|^{m}{\rm d}r.$$  By the scaling $r=\sqrt{2s}z$, we get
\begin{equation*}
    J_{m}\le C\int_{\mathbb{R}^n}e^{-\frac{|z|^{2}}{2}}s^{\frac{m}{2}}|z|^{m}{\rm d}z\leq C_{m}s^{\frac{m}{2}.}
\end{equation*}
Therefore,
\begin{equation*}
    I_{in}\leq C(s^{\frac{1}{2}}+s)\leq Cs^{\frac{1}{2}}=C\sqrt{1-e^{-t}}.
\end{equation*}
Substituting this estimate into \eqref{nabla w estimate 1}, we obtain
\begin{equation*}
|\nabla w(y)|
\le
C\int_0^\infty \frac{e^{-3t/2}}{\sqrt{1-e^{-t}}}
(1+e^{-t/2}|y|)^{-1}{\rm d}t.
\end{equation*}
Using
\begin{equation*}
    (1 + e^{-t/2}|y|)^{-1} \leq e^{\frac{t}{2}}(1 + |y|)^{-1},
\end{equation*}
it follows that
\begin{equation*}
|\nabla w(y)|\leq C(1+|y|)^{-1}\int_0^\infty \frac{e^{-t}}{\sqrt{1-e^{-t}}}
{\rm d}t.
\end{equation*}
The last integral is finite, since the integrand behaves like $t^{-1/2}$ as $t\rightarrow0^{+}$, and decays exponentially as $t\rightarrow+\infty$. Consequently,
\begin{equation*}
    |\nabla w(y)|\le C(1+|y|)^{-1},\quad \text{in }\mathbb{R}^{n}.
\end{equation*}
\end{proof}

To prove Theorem \ref{theorem 1}, we also need a result concerning the first eigenfunction of the operator
\begin{equation*}
\mathcal{L}\psi=-\Delta \psi + \frac{y}{2}\cdot \nabla \psi - e^{w}\psi.
\end{equation*}
\begin{lemma}\label{lem eigenfunction decay}
    Assume $w$ is a bounded non-constant solution of \eqref{stationary problem} satisfying \eqref{eq:w-decay}. Let $\lambda_{1}$ be the first eigenvalue of the eigenvalue problem
\begin{equation}\label{eq:eigenvalue p}
-\Delta \psi + \frac{y}{2}\cdot \nabla \psi - e^{w}\psi=\lambda\psi,\quad\text{in }\mathbb{R}^{n}.
\end{equation}
If $f$ is a positive eigenfunction associated to $\lambda_{1}$, then there exists a positive constant $C$ such that
\begin{equation*}
    (1 + |y|)^{\beta}|f| + (1 + |y|)^{\beta+1} |\nabla f| \leq C \quad \text{in }  \mathbb{R}^n,
\end{equation*}
where $0<\beta<-2(1+\lambda_{1})$.
\end{lemma}
\begin{proof}
    Let $\lambda_{1}$ be the smallest eigenvalue of \eqref{eq:eigenvalue p} and let $f$ be a positive eigenfunction associated to $\lambda_{1}$. Then $f$ satisfies
\begin{equation}\label{eq:eigenvalue f}
-\Delta f + \frac{y}{2}\cdot \nabla f - e^{w}f=\lambda_{1}f,\quad\text{in }\mathbb{R}^{n}.
\end{equation}
Since $w$ is a bounded non-constant solution of \eqref{stationary problem}, by Lemma \ref{lem not change sign} and Corollary \ref{cor eigenvalue problem}, $\lambda_{1}<-1$. Taking $\Phi(y)=|y|^{-\beta}$,  we have
\begin{equation}\label{eq:barrier}
    \mathcal{L}\Phi+\lambda_{1}\Phi=\left[-\beta(\beta+2-n)|y|^{-2}-\frac{\beta}{2}-e^{w}-\lambda_{1}\right]\Phi.
\end{equation}
Hence, \eqref{eq:w-decay} implies there exists a positive constant $R$ such that
\begin{equation}\label{eq:outer positive}
    -\beta(\beta+2-n)|y|^{-2}-\frac{\beta}{2}-e^{w}-\lambda_{1}>0,\quad \text{in }\mathbb{R}^{n}\setminus B_{R}.
\end{equation}
By taking $R$ large enough, we may also assume that
\begin{equation}\label{eq:zero term positive}
    -e^{w}-\lambda_{1}>0,\quad \text{for }|y|\ge R.
\end{equation}
By standard elliptic regularity theory, there exists a positive constant $M$ such that
\begin{equation}\label{eq:f interior}
    |f|\leq MR^{-\beta},\quad \text{in }B_{R}.
\end{equation}
Now for any $k\geq1$, we consider the Dirichlet problem:
\begin{equation}\label{eq:barrier equation}
\left\{
\begin{aligned}
&-\Delta g_k + \frac{y}{2} \cdot \nabla g_k - e^{w}g_k - \lambda_1 g_k = 0,
&& \text{in } B_{R+k} \setminus B_R, \\
&g_k = f,
&& \text{on } \partial B_R, \\
&g_k = M |y|^{-\beta},
&& \text{on } \partial B_{R+k}.
\end{aligned}
\right.
\end{equation}
Since we have assumed that \eqref{eq:zero term positive} holds, the zero order term of the second order elliptic equation in \eqref{eq:barrier equation} is positive. Using
\eqref{eq:barrier}, it follows from \eqref{eq:outer positive}, \cite[Theorem 8.3]{Gilbarg-T2001} and the maximum principle that for any $k\ge1$, \eqref{eq:barrier equation} has a unique smooth solution $g_{k}$, which is bounded above by $M|y|^{-\beta}$. By using the maximum principle again, $g_{k}$ is bounded below by 0. Letting $k\rightarrow\infty$, we get from the Arzel$\acute{\text{a}}$-Ascoli theorem that $\{g_{k}\}$ converges to some function $g_{\infty}$ in $C_{loc}^{2}\left(\mathbb{R}^{n}\setminus B_{R}\right)$, where
\begin{equation}\label{equation g wuqiong}
\left\{
\begin{aligned}
&-\Delta g_{\infty} + \frac{y}{2} \cdot \nabla g_{\infty} - e^{w}g_{\infty} - \lambda_1 g_{\infty} = 0,
&& \text{in } \mathbb{R}^n \setminus B_R, \\
&g_{\infty} = f,
&& \text{on } \partial B_R.
\end{aligned}
\right.
\end{equation}
Furthermore,
\begin{equation*}
|g_{\infty}| \leq M |y|^{-\beta}, \quad \text{in } \mathbb{R}^n \setminus B_R.
\end{equation*}
Multiplying both sides of \eqref{equation g wuqiong} by $g_{\infty}\rho$ over $\mathbb{R}^n \setminus B_R(0)$, we obtain
\begin{equation*}
\int_{\mathbb{R}^n \setminus B_R(0)} |\nabla g_{\infty}|^2 \rho {\rm d}y < \infty.
\end{equation*}

We claim that $g_{\infty}=f$.

Indeed, we denote by $h:=f-g_{\infty}$, then $h$ satisfies
\begin{equation}\label{equation h}
\left\{
\begin{aligned}
&-\Delta h + \frac{y}{2} \cdot \nabla h - e^{w} h -\lambda_1 h = 0,
&& \text{in } \mathbb{R}^n \setminus B_R, \\
&h = 0,
&& \text{on } \partial B_R.
\end{aligned}
\right.
\end{equation}
Let $\phi_{k}$ be a smooth cutoff function such that $\phi_{k}=1$ in $B_{k}$, $\phi_{k}=0$ outside $B_{2k}$  and $|\nabla\phi_{k}|\leq C/k$. Multiplying both sides of \eqref{equation h} by $h\phi_{k}^{2}\rho$ and integrating by parts, we derive
\begin{align*}
0 =&  \int_{\mathbb{R}^n \setminus B_R} |\nabla h|^2 \phi_k^2 \rho {\rm d}y
- \int_{\mathbb{R}^n \setminus B_R} e^{w}h^2 \phi_k^2 \rho {\rm d}y \\
&- \lambda_1 \int_{\mathbb{R}^n \setminus B_R} h^2 \phi_k^2 \rho {\rm d}y
+ 2 \int_{\mathbb{R}^n \setminus B_R} h \phi_k \nabla h \cdot \nabla \phi_k  \rho {\rm d}y.
\end{align*}
Recall that \eqref{eq:zero term positive}, then we have
\begin{equation*}
\int_{\mathbb{R}^n \setminus B_R} |\nabla h|^2 \phi_k^2 \rho {\rm d}y
\leq -2 \int_{\mathbb{R}^n \setminus B_R} h \phi_k \nabla h \cdot \nabla \phi_k  \rho {\rm d}y.
\end{equation*}
Letting $k\rightarrow\infty$, we deduce that
\begin{equation*}
\int_{\mathbb{R}^n \setminus B_R} |\nabla h|^2 \rho {\rm }dy = 0.
\end{equation*}
Hence $h=0$ and the claim is true.

We next prove the decay estimate for $\nabla f$. For each $i=1,\dots,n$, let
\[
f_i:=\partial_i f.
\]
Differentiating \eqref{eq:eigenvalue f} with respect to $y_i$, we obtain
\[
-\Delta f_i+\frac{y}{2}\cdot \nabla f_i-e^w f_i-\Bigl(\lambda_1-\frac12\Bigr)f_i
= e^w (\partial_i w) f,
\quad\text{in }\mathbb R^n.
\]
Since we have already proved
\[
|f(y)|\le C(1+|y|)^{-\beta},
\]
and by Lemma \ref{lem nabla and hessian w decay},
\[
|\nabla w(y)|\le C(1+|y|)^{-1},
\]
it follows that
\[
|e^w (\partial_i w)f|
\le C(1+|y|)^{-1}(1+|y|)^{-\beta}
\le C(1+|y|)^{-\beta-1}
\quad\text{in }\mathbb R^n.
\]

Now let $\Psi(y):=|y|^{-\beta-1}$,  we define
\[
\mathcal L_1
:=
-\Delta+\frac{y}{2}\cdot \nabla-e^w-\Bigl(\lambda_1-\frac12\Bigr),
\]
then
\[
\mathcal L_1 \Psi
=
\left[
-(\beta+1)(\beta+3-n)|y|^{-2}
-\frac{\beta}{2}
-e^w
-\lambda_1
\right]\Psi.
\]

Hence, \eqref{eq:w-decay} implies there exists a positive constant $R$ such that
\begin{equation*}
    -(\beta+1)(\beta+3-n)|y|^{-2}-\frac{\beta}{2}-e^{w}-\lambda_{1}>0,\quad \text{in }\mathbb{R}^{n}\setminus B_{R}.
\end{equation*}
By taking $R$ large enough, we may also assume that
\begin{equation}\label{eq:zero term positive1}
    -e^{w}-\lambda_{1}+\frac{1}{2}>0,\quad \text{for }|y|\ge R.
\end{equation}
By standard elliptic regularity theory, there exists a positive constant $M$ such that
\begin{equation}\label{eq:nabla f interior}
    |f_{i}|\leq MR^{-\beta-1},\quad \text{in }B_{R}.
\end{equation}
Moreover, by the definition of $\Psi(y)$, it is easy to verify
\begin{equation*}
    |e^w (\partial_i w)f|\leq C\Psi(y),\quad\text{in } \mathbb{R}^{n}\setminus B_{R}.
\end{equation*}
Therefore
\begin{equation}\label{eq:L1 Psi bigger than source}
    -\Delta \Psi + \frac{y}{2}\cdot \nabla \Psi - e^{w}\Psi - \left(\lambda_{1} - \frac{1}{2}\right)\Psi
\ge C\lvert e^{w} w_i f \rvert
\quad \text{in } \mathbb{R}^n \setminus B_R.
\end{equation}

We next prove $f_{i}\leq M\Psi(y)$. For any $k\geq1$ and each $i=1,\dots,n$, we consider the following Dirichlet problem:
\begin{equation}\label{eq:nabla f barrier equation for plus}
\left\{
\begin{aligned}
&-\Delta g_{i,k,+} + \frac{y}{2} \cdot \nabla g_{i,k,+} - e^{w}g_{i,k,+} - \left(\lambda_1-\frac{1}{2}\right) g_{i,k,+} = e^{w}w_{i}f,
&& \text{in } B_{R+k} \setminus B_R, \\
&g_{i,k,+} = f_i,
&& \text{on } \partial B_R, \\
&g_{i,k,+}= M |y|^{-\beta-1},
&& \text{on } \partial B_{R+k}.
\end{aligned}
\right.
\end{equation}
Using \eqref{eq:zero term positive1} and \cite[Theorem 8.3]{Gilbarg-T2001} again, we know \eqref{eq:nabla f barrier equation for plus} has a unique smooth solution $g_{i,k,+}$. Now define $l_{i,k,+}:=M\Psi(y)-g_{i,k,+}$. By taking $C\ge1$ in \eqref{eq:L1 Psi bigger than source} together with \eqref{eq:nabla f barrier equation for plus}, we obtain
\begin{equation*}
    \mathcal{L}_{1}l_{i,k,+}\ge C\lvert e^{w} w_i f \rvert -\lvert e^{w} w_i f \rvert\geq0.
\end{equation*}
On $\partial B_{R}$, \eqref{eq:nabla f interior} implies $l_{i,k,+}\geq0$, and on $\partial B_{R+k}$, $l_{i,k,+}=0$. Hence the maximum principle yields that
\begin{equation*}
    g_{i,k,+}\leq M\Psi(y),\quad\text{in } B_{R+k}\setminus B_{R}.
\end{equation*}
Similar to the above analysis, we have
\[
\left\{
\begin{aligned}
&-\Delta g_{i,\infty,+} + \frac{y}{2}\cdot \nabla g_{i,\infty,+} - e^{w}g_
{i,\infty,+}
- \left(\lambda_1 - \frac{1}{2}\right)g_{i,\infty}^{+}
= e^{w} w_i f
&&\text{in } \mathbb{R}^n \setminus B_R,\\
&\;\;g_{i,\infty,+} = f_i
&& \text{on } \partial B_R.
\end{aligned}
\right.
\]
Moreover, we have
\begin{equation*}
    g_{i,\infty,+} \le M\Psi(y) \qquad \text{in } \mathbb{R}^n \setminus B_R.
\end{equation*}
Likewise, one has
\begin{equation}\label{plus bounded}
     f_{i} \le M\Psi(y) \quad \text{in } \mathbb{R}^n \setminus B_R.
\end{equation}

To prove $-f_{i}\leq M\Psi(y)$, we consider the following problem :
\begin{equation*}
\left\{
\begin{aligned}
&-\Delta g_{i,k,-} + \frac{y}{2} \cdot \nabla g_{i,k,-} - e^{w}g_{i,k,-} - \left(\lambda_1-\frac{1}{2}\right) g_{i,k,-} = -e^{w}w_{i}f,
&& \text{in } B_{R+k} \setminus B_R, \\
&g_{i,k,-} = -f_i,
&& \text{on } \partial B_R, \\
&g_{i,k,-} = M |y|^{-\beta-1},
&& \text{on } \partial B_{R+k}.
\end{aligned}
\right.
\end{equation*}
By defining $l_{i,k,-}:=M\Psi(y)-g_{i,k,-}$ and using the similar analysis arguments,  we have for each $i=1,\dots,n$,
\begin{equation}\label{minus bounded}
    -f_{i} \le M\Psi(y) \quad \text{in } \mathbb{R}^n \setminus B_R.
\end{equation}
Combine  \eqref{eq:f interior}, \eqref{plus bounded} and \eqref{minus bounded} together, we conclude that
\begin{equation*}
    |\nabla f|\leq C(1+|y|)^{-\beta-1},\quad \text{in }\mathbb{R}^{n}.
\end{equation*}

\end{proof}

Next, we shall prove Theorem \ref{theorem 1}.
\begin{proof}[Proof of Theorem \ref{theorem 1}]
Let us take a one-parameter variation $w_{s}=w+sf$ for $s\in[-2\epsilon,2\epsilon], $ where $f$ is the first eigenfunction of the operator
\begin{equation*}
\mathcal{L}\psi=-\Delta \psi + \frac{y}{2}\cdot \nabla \psi - e^{w}\psi.
\end{equation*}
Without loss of generality, we may assume $f>0$ in $\mathbb{R}^{n}$. By the proof of Lemma \ref{lem F-stable}, we know that for any $x(s)$ and $t(s)$ with $x(0)=0$ and $t(0)=-1$,
\begin{equation*}
    \left.\frac{\partial^2}{\partial s^2}
\Big(F_{x(s),t(s)}(w_{s})\Big)\right|_{s=0}<0.
\end{equation*}
We will use this to prove that $w$ is also entropy unstable.

\noindent\textbf{Setting up the proof:} Define a function $G:\mathbb{R}^{n}\times\mathbb{R}^{-}\times[-2\varepsilon,2\varepsilon]$ by
\[
G(x_0,t_0,s):=F_{x_0,t_0}(w_s).
\]
We will show that there exists some $\epsilon_{1}>0$ so that if $s\neq0$ and $|s|\leq\epsilon_{1}$, then
\begin{equation}\label{lambda ws}
\lambda(w_s)= \sup_{x_0 \in \mathbb{R}^n,\; t_0 \in (-\infty,0)}G(x_0,t_0,s)
< G(0,-1,0)= \lambda(w).
\end{equation}
This will give the theorem with $w_{s}$ for any $s\neq0$ in $(-\epsilon_{1},\epsilon_{1})$.

The remainder of the proof is devoted to establishing \eqref{lambda ws}. The key point will be:
\begin{enumerate}
\item $G$ has a strict local maximum at $(0,-1,0)$.
\item $G(x_0,t_0,0)$ has a strict global maximum at $(0,-1)$.
\item $|\partial_s G|$ is bounded on compact sets of $\mathbb{R}^{n}\times\mathbb{R}^{-}\times[-2\varepsilon,2\varepsilon]$.
\item $G(x_0,t_0,s)$ is strictly less than $G(0,-1,0)$ whenever $|x_0|$ is sufficiently large.
\item $G(x_0,t_0,s)$ is strictly less than $G(0,-1,0)$ whenever $|\log (-t_0)|$ is sufficiently large.
\end{enumerate}
\noindent\textbf{The proof of \eqref{lambda ws} assuming (1)-(5):} Depending on the size of $|x_0|^2+|\log(-t_0)|^2$, we divide into three separate regions:

First, it follows from steps (4) and (5) that there is some $R>0$ so that \eqref{lambda ws} holds for every $s$ whenever
\begin{equation*}
    |x_0|^2+|\log(-t_0)|^2>R^{2}.
\end{equation*}

Second, as long as $s$ is small, step (1) implies \eqref{lambda ws} holds when $|x_0|^2+|\log(-t_0)|^2$ is sufficiently small.

Finally, in the intermediate region where $|x_0|^2+|\log(-t_0)|^2$ is bounded from above and bounded uniformly away from zero, step (2) says that $G$ is strictly less than $\lambda(w)$ as $s=0$ and step (3) implies the derivative of $G$ with respect to $s$ is uniformly bounded. Hence, there exists  $\epsilon_{3}>0$ so that $G(x_{0},t_{0},s)$ is strictly less than $\lambda(w)$ whenever $(x_{0},t_{0})$ is in the intermediate region as long as $|s|\leq\epsilon_{3}$.

This completes the proof of \eqref{lambda ws} under assuming (1)-(5).

\noindent\textbf{The verify  of (1)-(5):}

\noindent\textbf{Step (1):}
Since $w$ solves \eqref{stationary problem}, it follows from Lemma \ref{lem critical point} that $\nabla G$ vanishes at $(0,-1,0)$. We claim that the Hessian of $G$ at $(0,-1,0)$ is negative definite. Equivalently, for every nonzero direction $(y_0,-a,b)\in \mathbb R^n\times \mathbb R\times \mathbb R$, one has
\[
\left.\frac{d^2}{d\sigma^2}G\bigl(\sigma y_0,-(1+a\sigma),b\sigma\bigr)\right|_{\sigma=0}<0.
\]

To see this, we define
\[
\gamma(\sigma):=G\bigl(\sigma y_0,-(1+a\sigma),b\sigma\bigr)
=F_{x(\sigma),t(\sigma)}(w+\sigma \phi),
\]
where
\[
x(\sigma)=\sigma y_0,\quad t(\sigma)=-(1+a\sigma),\quad \phi=bf.
\]
Applying Lemma \ref{lem second varaition} with $x'(0)=y_0$, $t'(0)=-a$, and $\phi=bf$, we obtain
\begin{equation}\label{eq:gamma''}
    \begin{split}
        \gamma''(0)
&=-a\int_{\mathbb R^n} (\nabla w\cdot y+2)bf\rho{\rm d}y
+\int_{\mathbb R^n} |\nabla(bf)|^2\rho{\rm d}y
-\int_{\mathbb R^n} e^w (bf)^2\rho{\rm d}y \\
&\quad
-\int_{\mathbb R^n} (\nabla w\cdot y_0)bf\rho{\rm d}y
-\frac12\int_{\mathbb R^n} (\nabla w\cdot y_0)^2\rho{\rm d}y
-a^2\int_{\mathbb R^n}\left(1+\frac y2\cdot \nabla w\right)^2\rho{\rm d}y.
    \end{split}
\end{equation}

Since $f$ is the first eigenfunction of $\mathcal{L}$, we have
\[
\mathcal{L}f=\lambda_1 f.
\]
On the other hand, by Lemma \ref{lem eigenvalue problem},
\[
\mathcal{L}w_i=-\frac12 w_i,\quad \mathcal{L}(2+y\cdot \nabla w)=-(2+y\cdot \nabla w).
\]
Moreover, since $w$ is non-constant, Lemma \ref{lem not change sign} implies that $2+y\cdot \nabla w$ changes sign. Hence $-1$ cannot be the first eigenvalue, and therefore $\lambda_1<-1$. By the orthogonality of eigenfunctions corresponding to different eigenvalues, we get
\begin{equation}\label{eq: orthogonality 1}
    \int_{\mathbb R^n} (2+y\cdot \nabla w)f\rho{\rm d}y=0,
\end{equation}
and
\begin{equation}\label{eq: orthogonality 2}
\int_{\mathbb R^n} (\nabla w\cdot y_0)f\rho{\rm d}y
=0.
\end{equation}
Substituting \eqref{eq: orthogonality 1} and \eqref{eq: orthogonality 2} into \eqref{eq:gamma''}, we arrive at
\[
\gamma''(0)
=
\lambda_1 b^2\int_{\mathbb R^n} f^2\rho\,dy
-\frac12\int_{\mathbb R^n} (\nabla w\cdot y_0)^2\rho\,dy
-a^2\int_{\mathbb R^n}\left(1+\frac y2\cdot \nabla w\right)^2\rho\,dy.
\]

Note that  each term on the right-hand side of the above formula is non-positive. In fact,  each term is strictly negative whenever the corresponding parameter is nonzero:

\begin{itemize}
\item If $b\neq 0$, then since $\lambda_1< -1$ and $f>0$,
\[
\lambda_1 b^2\int_{\mathbb R^n} f^2\rho{\rm d}y<0.
\]

\item If $a\neq 0$, then
\[
-a^2\int_{\mathbb R^n}\left(1+\frac y2\cdot \nabla w\right)^2\rho{\rm d}y<0,
\]
because
\[
1+\frac y2\cdot \nabla w=\frac12(2+y\cdot \nabla w), \hbox{ and } 2+y\cdot \nabla w\not\equiv 0.
\]

\item If $y_0\neq 0$, then by Remark \ref{rem 1},
\[
-\frac12\int_{\mathbb R^n} (\nabla w\cdot y_0)^2\rho\,dy<0.
\]

\end{itemize}

Hence, for every nonzero $(y_0,-a,b)$, at least one of the three terms is strictly negative, while the others are non-positive. Therefore,
$
\gamma''(0)<0,
$
which implies that the Hessian of $G$ at $(0,-1,0)$ is negative definite. It follows that $G$ has a strict local maximum at $(0,-1,0)$, so there exists $\epsilon_{2}\in(0,\epsilon)$ so that
\begin{equation*}
G(x_0,t_0,s)<G(0,-1,0),\quad \text{if }0<|x_0|^{2}+|\log(-t_0)|^{2}+|s|^{2}\leq \epsilon_{2}^{2}.
\end{equation*}

\noindent\textbf{Step (2):} In this step, we will show that there exists a positive constant $\delta>0$ such that
\begin{equation}\label{eq:gap-s0}
G(x_{0},t_{0},0)<G(0,-1,0)-\delta,
\end{equation}
for all $x_{0},t_{0}$ with $\frac{\epsilon_{2}^{2}}{4}<|x_{0}|^{2}+|\log(-t_{0})|^{2}$.

Indeed, we get from the definition of $G(x_{0},t_{0},s)$ and Lemma \ref{lem entropy attainable} that  $G(x_{0},t_{0},0)=F_{x_{0},t_{0}}(w)$ and $G(0,-1,0)=F_{0,-1}(w)=\lambda(w)$. Thus \eqref{eq:gap-s0} is a direct consequence of Lemma \ref{lem F-stable}.

\medskip
\noindent\textbf{Step (3):} $|\partial_{s}G|$ is bounded on compact sets of $\mathbb{R}^{n}\times\mathbb{R}^{-}\times[-2\varepsilon,2\varepsilon]$.

By the definition of $G$ and the first variation formula, we have
\begin{equation*}
\begin{split}
\partial_s G(x_0,t_0,s)=&(-t_0)\int_{\mathbb{R}^n}\nabla(w+sf)\cdot\nabla fG(y-x_0,t_0){\rm d}y\\
&\;-(-t_0)\int_{\mathbb{R}^n}e^{w+sf}fG(y-x_0,t_0){\rm d}y
+\int_{\mathbb{R}^n}fG(y-x_0,t_0){\rm d}y.
\end{split}
\end{equation*}
It is clear that $\partial_{s}G$ is continuous in all three variables $x_{0}$, $t_{0}$ and s. We deduce that $\partial_{s}G$ is bounded on compact subsets.

\medskip
\noindent\textbf{Step (4):} We fix a positive constant $M$ such that $M\gg1$. For any $x_{0}$ with $|x_{0}|\ge M$, we will show
\begin{equation}\label{Gs<G0}
    G(x_0,t_0,s)<G(0,-1,0),
\end{equation}
namely
\begin{equation}\label{Gs-G0<0 expand}
\begin{split}
& \frac{1}{2}(-t_0)\int_{\mathbb{R}^n} |\nabla (w+sf)|^2 G(y-x_0,t_0){\rm d}y
- (-t_0)\int_{\mathbb{R}^n} e^{w+sf} G(y-x_0,t_0){\rm d}y  \\
&\quad + \int_{\mathbb{R}^n} (w+sf) G(y-x_0,t_0){\rm d}y
+ \log(-t_0)- \frac12\int_{\mathbb{R}^n}|\nabla w|^2\rho{\rm d}y
+\int_{\mathbb{R}^n}e^w\rho{\rm d}y\\
&\quad-\int_{\mathbb{R}^n}w\rho{\rm d}y<0.
\end{split}
\end{equation}
By step (2), in order to prove \eqref{Gs<G0}, it suffices to show that
\begin{equation*}
     G(x_0,t_0,s)\leq G(x_0,t_0,0)+C|s|,
\end{equation*}
namely
\begin{equation}\label{linear control equi}
\begin{split}
&s(-t_0)\int_{\mathbb{R}^n} \nabla w\cdot\nabla f G(y-x_0,t_0){\rm d}y+\frac{s^{2}}{2}(-t_0)\int_{\mathbb{R}^n} |\nabla f|^2 G(y-x_0,t_0){\rm d}y\\
 &\quad-(-t_0)\int_{\mathbb{R}^n} (e^{w+sf} -e^{w})G(y-x_0,t_0){\rm d}y
+s\int_{\mathbb{R}^n} fG(y-x_0,t_0){\rm d}y\leq C|s|.
\end{split}
\end{equation}

Suppose that $\lvert-t_{0}\lvert$ is bounded above and bounded away from zero. By Lemmas \ref{lem nabla and hessian w decay} and \ref{lem eigenfunction decay} together with the integrability of Gaussian kernel $G$, we know that \eqref{linear control equi} holds.

We now turn to the case where $\lvert -t_{0}\lvert$ is sufficiently large. We first estimate the gradient term in \eqref{Gs-G0<0 expand}. Expanding the square, we have
\begin{equation}\label{eq:gradient term in Gs}
    \begin{split}
     &\frac{1}{2}(-t_0) \int_{\mathbb{R}^n} |\nabla(w+sf)|^2 G(y-x_0, t_0) \mathrm{d}y \\
    ={}& \frac{1}{2}(-t_0) \int_{\mathbb{R}^n} |\nabla w|^2 G(y-x_0, t_0) \mathrm{d}y \\
    &+ s(-t_0) \int_{\mathbb{R}^n} \nabla w \cdot \nabla f \, G(y-x_0, t_0) \mathrm{d}y \\
    &+ \frac{s^2}{2}(-t_0) \int_{\mathbb{R}^n} |\nabla f|^2 G(y-x_0, t_0) \mathrm{d}y.
    \end{split}
\end{equation}
We now estimate the second term on the RHS of \eqref{eq:gradient term in Gs}. Take $M>1$ sufficiently large. For any $|x_{0}|\ge M$, we consider the change of variable $y=x_{0}+\sqrt{-t_{0}}z$. Define three regions in $\mathbb{R}^{n}$ by
\[
\Omega_1=\left\{ z\in\mathbb{R}^n:\ |z|\le \frac{|x_0|}{2\sqrt{-t_0}} \right\},\qquad
\Omega_2=\left\{ z\in\mathbb{R}^n:\ \frac{|x_0|}{2\sqrt{-t_0}}\le |z|\le \frac{3|x_0|}{2\sqrt{-t_0}} \right\},
\]
\[
\Omega_3=\left\{ z\in\mathbb{R}^n:\ |z|> \frac{3|x_0|}{2\sqrt{-t_0}} \right\}.
\]
Then we have
\begin{equation*}
\begin{aligned}
&\quad(-t_0) \int_{\mathbb{R}^n} \nabla w \cdot \nabla f\, G(y-x_0,t_0){\rm d}y\\
&= (-t_0)\int_{\mathbb{R}^n}
\nabla w \cdot \nabla f\!\left(x_0+\sqrt{-t_0}\,z\right)\rho(z){\rm d}z \\
&\le I_{1}+I_{2}+I_{3}
\end{aligned}
\end{equation*}
with
\begin{equation*}
\begin{aligned}
I_{1} &= (-t_0)\int_{\Omega_1} |\nabla w||\nabla f|\bigl(x_0+\sqrt{-t_0}z\bigr)\rho(z){\rm d}z,\\
I_{2} &= (-t_0)\int_{\Omega_2} |\nabla w||\nabla f|\bigl(x_0+\sqrt{-t_0}z\bigr)\rho(z){\rm d}z,\\
I_{3} &= (-t_0)\int_{\Omega_3} |\nabla w||\nabla f|\bigl(x_0+\sqrt{-t_0}z\bigr)\rho(z){\rm d}z.
\end{aligned}
\end{equation*}
First, we treat the term $I_{3}$. If $|z|\ge \frac{3|x_0|}{2\sqrt{-t_0}}$, then
\begin{equation}\label{scaling}
\begin{aligned}
\left|x_0+\sqrt{-t_0}z\right|&\ge \sqrt{-t_0}|z|-|x_0| \\
&\ge \sqrt{-t_0}|z|-\frac{2}{3}\sqrt{-t_0}|z| \\
&= \frac{1}{3}\sqrt{-t_0}|z|\ge\frac{|x_{0}|}{2}\ge\frac{M}{2}.
\end{aligned}
\end{equation}
By \eqref{eq:w-decay1} and Lemmas \ref{lem nabla and hessian w decay} and \ref{lem eigenfunction decay}, there exists a positive constant $C$ such that
\begin{equation*}
\begin{aligned}
I_{3}
&= (-t_0)
\int_{\Omega_{3}}
|\nabla w||\nabla f|(x_0+\sqrt{-t_0}z)\rho(z){\rm d}z \\
&\le C(-t_0)
\int_{\Omega_{3}}
|x_0+\sqrt{-t_0}z|^{-2}\rho(z){\rm d}z.
    \end{aligned}
\end{equation*}
Together with \eqref{scaling}, we derive that there exists a positive constant $C$ such that
\begin{equation}\label{3 estimate}
\begin{aligned}
I_{3}
&\le C \int_{\Omega_{3}}
|z|^{-2}\,\rho(z){\rm d}z \\
&\le C \int_{\mathbb{R}^n}
|z|^{-2}\rho(z){\rm d}z
\le C .
\end{aligned}
\end{equation}

We next deal with term $I_{1}$, we divide the estimate into two cases.

 If $\dfrac{|x_0|}{\sqrt{-t_0}} \ge M$, then we have
\begin{equation}\label{I1 big}
I_{1} \le C(-t_0)|x_0|^{-2} \le C.
\end{equation}
On the other hand, if $\dfrac{|x_0|}{\sqrt{-t_0}} \le M$, then we have
\begin{equation}\label{I1 small}
I_{1}
\le C(-t_0)|x_0|^{-2}
\left(\frac{|x_0|}{\sqrt{-t_0}}\right)^n
\le C .
\end{equation}
Finally, we consider the term $I_{2}$, we have
\begin{equation}\label{I2}
\begin{aligned}
I_{2}
&= (-t_0)
\int_{\Omega_{2}}
|\nabla w||\nabla f|(x_0+\sqrt{-t_0}z)\rho(z){\rm d}z \\
&\le C(-t_0)
e^{\frac{|x_0|^2}{16t_0}}
\int_{\Omega_{2}}
|\nabla w||\nabla f|(x_0+\sqrt{-t_0}z){\rm d}z \\
&\le C(-t_0)
e^{\frac{|x_0|^2}{16t_0}}
(-t_0)^{-\frac{n}{2}}
\int_{\{|y|\le 4|x_0|\}} |\nabla w||\nabla f|{\rm d}y \\
&\le C(-t_0)
e^{\frac{|x_0|^2}{16t_0}}
(-t_0)^{-\frac{n}{2}}
|x_0|^{n-2} \\
&\le C
\left(\frac{|x_0|}{\sqrt{-t_0}}\right)^{n-2}
e^{\frac{|x_0|^2}{16t_0}}
\le C .
\end{aligned}
\end{equation}
By \eqref{3 estimate}-\eqref{I2}, we conclude that there exist $M\ge1$ and $C$ depending on $n$, $w$ and $M$ such that if $|x_{0}|\ge M$ and sufficiently large $\lvert-t_{0}\lvert$, then
\begin{equation*}
(-t_0)
\int_{\mathbb{R}^n}
|\nabla w||\nabla f|
G(y-x_0,t_0){\rm d}y
\le C .
\end{equation*}

Using a similar argument, together with the fact that $s\in[-2\epsilon,2\epsilon]$,  we can show that both the first and third terms on the RHS of \eqref{eq:gradient term in Gs} are bounded, which implies
\begin{equation*}
    \frac{1}{2}(-t_0) \int_{\mathbb{R}^n} |\nabla(w+sf)|^2 G(y-x_0, t_0) \mathrm{d}y\leq C.
\end{equation*}

We now estimate the second term on the RHS of \eqref{Gs-G0<0 expand}. Since $w+sf$ is bounded in $\mathbb{R}^{n}$, there exists a constant $\hat{C}$ such that
\begin{equation*}
    |w+sf|\le \hat{C} \quad \text{in } \mathbb R^n.
\end{equation*}
Then we have
\begin{equation*}
    e^{w+sf}\ge e^{-\hat{C}} \quad \text{in } \mathbb R^n.
\end{equation*}
By the normalization of the Gaussian kernel
\begin{equation*}
    \int_{\mathbb{R}^n} G(y - x_0, t_0)  {\rm d}y = 1,\quad \text{for any }x_{0}\in\mathbb{R}^{n},\quad t_{0}<0,
\end{equation*}
we get
\begin{equation*}
\int_{\mathbb R^n} e^{w+sf}G(y-x_0,t_0){\rm d}y \ge e^{-\hat{C}},
\end{equation*}
which means
\begin{equation*}
-(-t_0)\int_{\mathbb R^n} e^{w+sf}G(y-x_0,t_0){\rm d}y
\le Ct_0,\quad \text{for all sufficiently negative }\;t_{0}.
\end{equation*}

On the other hand, by Lemmas \ref{lem nabla and hessian w decay} and \ref{lem eigenfunction decay}, all the integrands involved in the remaining terms of \eqref{Gs-G0<0 expand}, apart from the logarithmic term $\log(-t_{0})$, are bounded. Since the Gaussian kernel is integrable, it follows that these terms are uniformly bounded and independent of $t_{0}$. Denote by $I(t_{0})$ the LHS of \eqref{Gs-G0<0 expand}. Then, for all sufficiently negative $t_{0}$, there exist two positive constants $C_1$ and $c$ such that
\begin{equation}\label{Gs expand rewrite}
   I(t_{0})\leq C_{1}+ct_{0}+\log(-t_{0})<0.
\end{equation}
This completes the proof of \eqref{Gs expand rewrite} for  sufficiently large $\lvert-t_{0}\lvert$ and sufficiently large $|x_{0}|$.
We postpone the discussion of the case $\lvert-t_{0}\lvert\rightarrow0$ to step (5).

\noindent\textbf{Step (5):} By step (4), there exists some $M>0$ so that if $|x_{0}|\geq M$ and $\lvert-t_{0}\lvert$ is sufficiently large, then $G(x_{0},t_{0},s)$ is strictly less than $\lambda(w)$. We now consider the case where $|x_{0}|\leq M$ and $\lvert-t_{0}\lvert$ is sufficiently large.

We first show that there exists a positive constant $C$ such that
\begin{equation*}
(-t_0) \int_{\mathbb{R}^n} |\nabla w|\cdot|\nabla f| G(y-x_0,t_0){\rm d}y\le C.
\end{equation*}
There are three cases:

\textbf{Case 1:}\;$2< n<3+\beta$. By Lemmas \ref{lem nabla and hessian w decay} and \ref{lem eigenfunction decay}, we obtain
\begin{align*}
&\quad(-t_0)\int_{\mathbb{R}^n}|\nabla w||\nabla f|G(y-x_0,t_0){\rm d}y\\
&\leq C(-t_0)^{1-\frac{n}{2}}\int_{\mathbb{R}^n}(1+|y|)^{-3-\beta}{\rm d}y \\
&\leq C(-t_0)^{1-\frac{n}{2}}\leq C.
\end{align*}

\textbf{Case 2:}\;$n>3+\beta$. In this case, we have
\begin{align*}
&\quad(-t_0) \int_{\mathbb{R}^n} |\nabla w||\nabla f|G(y-x_0,t_0){\rm d}y \\
&\leq C(-t_0)^{1-\frac{n}{2}} \int_{\{|y|\leq(-t_0)^{\alpha(\beta)}\}} (1+|y|)^{-3-\beta}{\rm d}y \\
&\quad + C(-t_0)\int_{\{|y|\geq(-t_0)^{\alpha(\beta)}\}} (-4\pi t_0)^{-\frac{n}{2}} e^{\frac{|y-x_0|^2}{4t_0}}{\rm d}y \leq C,
\end{align*}
where $\alpha(\beta):=\frac{n-2}{2(n-3-\beta)}$.

\textbf{Case 3:}\;$n=3+\beta$. By choosing $\alpha(\beta)>1/2$, this case can be proved in the same way as in Case 2.

Using arguments similar to those in cases 1-3, we deduce that
\begin{equation*}
    \frac{1}{2}(-t_0) \int_{\mathbb{R}^n} |\nabla(w+sf)|^2 G(y-x_0, t_0) \mathrm{d}y\leq C.
\end{equation*}

Next, following the same approach as in step (4), we show that \eqref{Gs-G0<0 expand} holds for $|x_{0}|\leq M$ and $\lvert-t_{0}\lvert$ sufficiently large.

Finally, we consider the case where $\lvert-t_{0}\lvert\rightarrow0$. Indeed, since $|\nabla w|$, $|\nabla f|$, $f$ are bounded in $\mathbb{R}^{n}$ and $w$ satisfies \eqref{eq:w-decay}, there exists a constant $C_{2}$ independent of $x_{0}$ such that
\begin{equation*}
    I(t_{0})\leq C_{2}+\log(-t_{0})<0,\quad \text{for sufficiently small }\; \lvert-t_{0}\lvert,
\end{equation*}
which means \eqref{Gs<G0}.
\end{proof}
\begin{remark}\label{rem 1}
Under the assumptions of Theorem 1.1, for every nonzero vector $y_0\in \mathbb R^n$, one has
\[
\nabla w\cdot y_0\not\equiv 0 \quad \text{in } \mathbb R^n.
\]
Equivalently,
\[
\int_{\mathbb R^n} (\nabla w\cdot y_0)^2\rho{\rm d}y>0.
\]

Indeed, if there exist some $y_0\neq 0$ such that
\[
\int_{\mathbb R^n} (\nabla w\cdot y_0)^2\rho{\rm d}y=0,
\]
then it follows from $\rho>0$ that
\[
\nabla w\cdot y_0\equiv 0 \quad \text{in } \mathbb R^n.
\]
Hence, for every fixed $y\in \mathbb R^n$, the function
\[
\psi(\mu):=w(y+\mu y_0)
\]
satisfies
\[
\psi'(\mu)=\nabla w(y+\mu y_0)\cdot y_0=0,
\]
which means $\psi$ is constant. Therefore,
\[
w(y+\mu y_0)\equiv w(y)\quad \text{for all }\mu\in \mathbb R.
\]
Namely, $w$ is invariant in the $y_0$-direction.

On the other hand, since $y_0\neq 0$, we have $|y+\mu y_0|\to \infty$ as $|\mu|\to\infty$. By the decay assumption \eqref{eq:w-decay}, we obtain
\begin{equation*}
    |w(y)|=|w(y+\mu y_0)|
\le C(1+|y+\mu y_0|)^{-1}\to 0, 
\qquad \text{as }|\mu|\to\infty.
\end{equation*}
Thus $w(y)=0$ for every $y\in \mathbb R^n$, namely $w\equiv 0$, which contradicts the assumption that $w$ is non-constant.
\end{remark}
\section{Rigidity result}
In this section, we prove the rigidity result. Theorem \ref{theorem 2} is motivated by the analogous result for self-shrinkers in mean curvature flow in Colding-Ilmanen-Minicozzi \cite{Colding-I-M2015}. The proof follows the iteration argument developed in \cite{Wang-W-W2025 Liouville}, which relies on the following two propositions.
\begin{proposition}\label{prop 1}
There exists $\varepsilon_{0}>0$ and $R_{1}$ such that for any solution $w$ of \eqref{stationary problem}, if $R>R_{1}$, $w$ is smooth in $\overline{B_{R}(0)}$ and
\begin{equation*}
\mathrm{dist}_R(w,0) + \mathrm{dist}_R\left(\Lambda(w), 2\right)
\le \varepsilon_{0},
\end{equation*}
then
\begin{equation*}
\mathrm{dist}_{R-3}(w,0) + \mathrm{dist}_{R-3}\left(\Lambda(w), 2\right)
\le \frac{\varepsilon_{0}}{10}.
\end{equation*}
\end{proposition}
\begin{proposition}\label{prop 2}
There exist two universal constants $R_{2}$ and $\theta$ such that for any solution $w$ of \eqref{stationary problem}, if $R>R_{2}$, $w$ is smooth in $\overline{B_{R}(0)}$ and
\begin{equation}\label{eq condition in prop2}
\mathrm{dist}_R(w,0) + \mathrm{dist}_R\left(\Lambda(w), 2\right)
\le \frac{\varepsilon_{0}}{10},
\end{equation}
then $w$ is smooth in $\overline{B_{(1+\theta)R}(0)}$ and
\begin{equation*}
\mathrm{dist}_{(1+\theta)R}(w,0) + \mathrm{dist}_{(1+\theta)R}\left(\Lambda(w), 2\right)
\le \varepsilon_{0},
\end{equation*}
where $\varepsilon_{0}$ is given in Proposition \ref{prop 1}.
\end{proposition}

According to Proposition \ref{prop 1}, establishing a worse bound on a sufficiently large ball $B_{R}(0)$ leads to a better bound on a slightly smaller ball $B_{R-3}(0)$. Conversely, Proposition \ref{prop 2} reveals that a better bound on a large ball $B_{R}(0)$ implies a worse bound on $B_{(1+\theta)R}(0)$ for some fixed $\theta>0$. Importantly, the constant $\theta$ does not depend on $R$.

To prove Proposition \ref{prop 1}, we need several lemmas.
\begin{lemma}\label{lem large scale control gives local rigidity}
    For any $R_{0}>1$ and $\varepsilon>0$, there exists $R_{1}\ge R_{0}$ such that if $R>R_{1}$, the solution of \eqref{stationary problem} $w$ is smooth in $B_{R}(0)$ and it satisfies
\begin{equation*}
\mathrm{dist}_R(w,0) + \mathrm{dist}_R\left(\Lambda(w), 2\right)
\le \varepsilon_{0},
\end{equation*}
where $\varepsilon_{0}$ is defined in Proposition \ref{prop 1}, then
\begin{equation*}
    \mathrm{dist}_{R_{0}}(w,0) + \mathrm{dist}_{R_{0}}\left(\Lambda(w), 2\right)
    \le \varepsilon.
\end{equation*}
\end{lemma}
\begin{proof}
We argue by contradiction. Suppose that there exist $R_{0}$ and $\varepsilon$, and a sequence of smooth solutions $w_{i}$ to \eqref{stationary problem} in $B_{R_{i}}(0)$ with $R_{i}\rightarrow+\infty$, such that
\begin{equation}\label{eq:Ri dist}
\mathrm{dist}_{R_{i}}(w_{i},0) + \mathrm{dist}_{R_{i}}\left(\Lambda(w_{i}), 2\right)
\le \varepsilon_{0},
\end{equation}
but
\begin{equation}\label{eq:contradict R0 dist}
    \mathrm{dist}_{R_{0}}(w_{i},0) + \mathrm{dist}_{R_{0}}\left(\Lambda(w_{i}), 2\right)
    > \varepsilon.
\end{equation}
By \eqref{eq:Ri dist} and standard elliptic regularity theory, there exists a smooth solution  $w_{\infty}$ of \eqref{stationary problem} such that $w_{i}\rightarrow w_{\infty}$ in $C_{loc}^{\infty}(\mathbb{R}^{n})$. Passing to the limit in \eqref{eq:Ri dist} gives
\begin{equation*}
    |w_{\infty}-0|+|\Lambda(w_{\infty})-2|\leq\varepsilon_{0}.
\end{equation*}
Then by Lemma \ref{lem not change sign}, $w_{\infty}\equiv0$.

On the other hand, the $C_{loc}^{\infty}(\mathbb{R}^{n})$ convergence applied to \eqref{eq:contradict R0 dist} implies
\begin{equation*}
    \mathrm{dist}_{R_{0}}(w_{\infty},0) + \mathrm{dist}_{R_{0}}\left(\Lambda(w_{\infty}), 2\right)
    \ge \varepsilon,
\end{equation*}
which clearly contradicts $w_{\infty}\equiv0$.
\end{proof}

The above lemma shows that $w$ can be made arbitrarily close to $0$ on any fixed compact set. To complete the proof of Proposition \ref{prop 1}, it suffices to prove that both $\nabla w$ and $\nabla\Lambda(w)$ are exponentially small in $R$. For this purpose, we introduce
\begin{equation*}
    \tau:=\frac{e^{w}}{\Lambda(w)}.
\end{equation*}
Since $e^{w}$ and $\Lambda(w)$ are continuous and positive on $B_{R}(0)$, this function is well defined. This auxiliary function $\tau$ satisfies the following equations.
\begin{lemma}\label{lem tau equation}
The function $\tau$ satisfies
\begin{equation}\label{tau eq1}
\frac{1}{\Lambda(w)^{2}\rho}\operatorname{div}\left(\Lambda(w)^{2}\rho\nabla\tau\right)
= \frac{e^{w}|\nabla w|^{2}}{\Lambda(w)}.
\end{equation}
\begin{equation}\label{tau eq2}
\frac{1}{\Lambda(w)^{2}\rho}\operatorname{div}\left(\Lambda(w)^{2}\rho\nabla(\tau^{2})\right)
= 2|\nabla \tau|^{2} + 2\tau \frac{e^{w}|\nabla w|^{2}}{\Lambda(w)}.
\end{equation}
\end{lemma}
\begin{proof}
    Since $w$ is a solution of \eqref{stationary problem}, the function $e^{w}$ satisfies
\begin{equation*}
\Delta(e^{w})-\frac{1}{2}y\cdot\nabla(e^{w})+(e^{w}-1)e^{w}
= e^{w}|\nabla w|^{2}.
\end{equation*}
It follows from Lemma \ref{lem eigenvalue problem} that
\begin{equation}\label{eq:L Lambda}
\Delta \Lambda(w) - \frac{1}{2} y \cdot \nabla \Lambda(w)
+ (e^{w} - 1)\Lambda(w) = 0.
\end{equation}
By a direct computation, we obtain
\begin{align*}
\frac{1}{\Lambda(w)^{2}\rho}\,
\operatorname{div}\!\left(\Lambda(w)^{2}\rho\nabla\tau\right)
&= \Delta\tau
 + \frac{2}{\Lambda(w)}\,\nabla\Lambda(w)\cdot\nabla\tau
 - \frac{1}{2}y\cdot\nabla\tau \\[6pt]
&= \frac{\Lambda(w)\Delta(e^{w})-e^{w}\Delta\Lambda(w)}{\Lambda(w)^{2}}
 - \frac{1}{2}y\cdot\nabla\tau \\[6pt]
&= \Lambda(w)^{-1}
   \Bigl[\tfrac12 y\cdot\nabla(e^{w})
   -(e^{w}-1)e^{w}
   +e^{w}\lvert\nabla w\rvert^{2}\Bigr]  \\
&\quad
 -\frac{e^{w}}{\Lambda(w)^{2}}
   \Bigl[\tfrac12 y\cdot\nabla\Lambda(w)
   -(e^{w}-1)\Lambda(w)\Bigr]
 -\frac12 y\cdot\nabla\tau \\[6pt]
&= \frac{e^{w}\lvert\nabla w\rvert^{2}}{\Lambda(w)},
\end{align*}
which is \eqref{tau eq1}.

Since
\begin{equation*}
\frac{1}{\Lambda(w)^{2}\rho}
\operatorname{div}\left(\Lambda(w)^{2}\rho\nabla (\tau^{2})\right)
=
\frac{2\tau}{\Lambda(w)^{2}\rho}\,
\operatorname{div}\left(\Lambda(w)^{2}\rho\nabla \tau\right)
+2|\nabla\tau|^{2},
\end{equation*}
\eqref{tau eq2} follows from \eqref{tau eq1}.
\end{proof}

Using Lemma \ref{lem tau equation}, we obtain the following integral estimate for $|\nabla w|$ and $|\nabla\Lambda(w)|$.

\begin{lemma}\label{lem nabla w and Lambda integral estimate}
    For any $s\in(0,R)$, we have
\begin{equation}\label{eq:nabla w and Lambda integral estimate}
\int_{B_{R-s}(0)}
\left(
|\nabla \tau|^{2}
+
\frac{\tau e^{w}|\nabla w|^{2}}{\Lambda(w)}
\right)
\Lambda(w)^{2}\rho{\rm d}y
\le
\frac{C}{s^{2}} R^{n} e^{-(R-s)^{2}/4}
\sup_{y\in B_{R}(0)} e^{2w(y)},
\end{equation}
where $C$ is a positive constant independent of $R$.
\end{lemma}
\begin{proof}
By \eqref{tau eq2},  one has
\begin{equation}\label{nabla tau square eq}
\frac{1}{\Lambda^{2}\rho}
\operatorname{div}\left(\Lambda(w)^{2}\rho\nabla(\tau^{2})\right)
=2|\nabla\tau|^{2}+2\tau^{2}|\nabla w|^{2}\quad \text{in } B_R(0).
\end{equation}
For any $\phi\in C_{0}^{\infty}(B_{R}(0))$, multiplying \eqref{nabla tau square eq} by $\phi^{2}\Lambda(w)^{2}\rho$ and integrating over $B_{R}(0)$, we obtain
\begin{equation*}
\int_{B_R(0)} \phi^2 \operatorname{div}\left(\Lambda(w)^2 \rho \nabla(\tau^2)\right) {\rm d}y
=
2 \int_{B_R(0)} \phi^2 \left(|\nabla \tau|^2 + \tau^2 |\nabla w|^2\right)\Lambda^2 \rho {\rm d}y .
\end{equation*}
Since $\phi\in C_{0}^{\infty}(B_{R}(0))$, we infer
\begin{equation*}
2\int_{B_R(0)} \phi^2
\left(
|\nabla \tau|^2 + \tau^2 |\nabla w|^2
\right)
\Lambda(w)^2 \rho{\rm d}y
=
-4\int_{B_R(0)} \phi\tau\nabla\phi\cdot\nabla\tau\,
\Lambda(w)^2 \rho{\rm d}y.
\end{equation*}
Applying the Cauchy-Schwarz inequality $4|\phi \tau \nabla \phi \cdot \nabla \tau |
\le\phi^2 |\nabla \tau|^2 + 4\tau^2 |\nabla \phi|^2$, we obtain
\begin{equation*}
2 \int_{B_R(0)} \phi^2 \left(|\nabla \tau|^2 + \tau^2 |\nabla w|^2\right) \Lambda(w)^2 \rho {\rm d}y
\le
\int_{B_R(0)} \phi^2 |\nabla \tau|^2 \Lambda(w)^2 \rho {\rm d}y
+
4 \int_{B_R(0)} \tau^2 |\nabla \phi|^2 \Lambda(w)^2 \rho {\rm d}y .
\end{equation*}
Hence
\begin{equation*}
\int_{B_R(0)} \phi^2 \left(|\nabla \tau|^2 + \tau^2 |\nabla w|^2\right)
\Lambda(w)^2 \rho {\rm d}y
\le
4 \int_{B_R(0)} \tau^2 |\nabla \phi|^2 \Lambda(w)^2 \rho {\rm d}y .
\end{equation*}
Choose $\phi$ satisfying $\phi\equiv1$ in $B_{R-s}(0)$ and $|\nabla\phi|\leq\frac{2}{s}$ in $B_{R}(0)\setminus B_{R-s}(0)$. Combining the definition of $\tau$,  we have
\begin{equation*}
\int_{B_{R-s}(0)}
\left(
|\nabla \tau|^{2}
+
\frac{\tau e^{w}|\nabla w|^{2}}{\Lambda(w)}
\right)
\Lambda(w)^{2}\rho{\rm d}y
\le
\frac{C}{s^{2}} R^{n} e^{-(R-s)^{2}/4}
\sup_{y\in B_{R}(0)} e^{2w(y)},
\end{equation*}
which is \eqref{eq:nabla w and Lambda integral estimate}.
\end{proof}

As a consequence of Lemma \ref{lem nabla w and Lambda integral estimate} and standard elliptic estimates, we derive the following pointwise bound for $|\nabla w|$ and $|\nabla\Lambda(w)|$.
\begin{lemma}\label{lem pointwise nabla w and Lambda estimate}
    There exists a universal constant $C$ such that
\begin{equation}\label{eq:nabla w and nabla Lambda estimate}
\sup_{y\in B_{R-3}(0)}
\left(
|\nabla \Lambda(w)(y)|^2 + |\nabla w(y)|^2
\right)
\le
C R^{2n} e^{-\frac{R}{4}}.
\end{equation}
\end{lemma}
\begin{proof}
    Taking $s=\frac{1}{2}$ in \eqref{eq:nabla w and Lambda integral estimate}, we arrive at
\begin{equation}\label{eq:s=1/2 estimate}
\int_{B_{R-\frac12}(0)}
\left(
|\nabla \tau|^2 + \tau^2 |\nabla w|^2
\right)
\Lambda(w)^2 \rho{\rm d}y
\le
C R^n \exp\left(-\frac{(R-\frac12)^2}{4}\right).
\end{equation}
It is easy to verify that
\begin{equation*}
\nabla \tau
=
\frac{e^w}{\Lambda(w)} \nabla w
-
\frac{e^w}{\Lambda(w)^2} \nabla \Lambda(w)
=
\tau \nabla w
-
\frac{\tau}{\Lambda(w)} \nabla \Lambda(w).
\end{equation*}
Since $w$ is bounded and $\Lambda(w)$ is bounded above and below away from zero, it follows that $\tau$ is also bounded above and below away from zero. Thus there exists a universal constant $C$ such that
\begin{equation}\label{eq:nabla Lambda}
    |\nabla\tau|^{2}\leq C(|\nabla\Lambda(w)|^{2}+|\nabla w|^{2}).
\end{equation}
Substituting \eqref{eq:nabla Lambda} into \eqref{eq:s=1/2 estimate}, we have
\begin{equation*}
\int_{B_{R-\frac{1}{2}}(0)}
\left(
|\nabla \Lambda(w)|^2 + |\nabla w|^2
\right)
\rho {\rm d}y
\le
C R^n
\exp\!\left(
-\frac{(R-\frac{1}{2})^2}{4}
\right).
\end{equation*}
Since
\begin{equation*}
e^{-\frac{|x|^2}{4}}
\ge
\exp\left\{
-\frac{R^2-2R+1}{4}
\right\},
\quad \text{on } B_{R-1}(0),
\end{equation*}
then we can obtain
\begin{equation}\label{eq:integral decay}
\int_{B_{R-1}(0)}
\left(
|\nabla \Lambda(w)|^2 + |\nabla w|^2
\right){\rm d}y
\le
C R^n \exp^{-\frac{R}{4}},
\end{equation}
which is the integral decay estimate on $\nabla\Lambda(w)$ and $\nabla w$.

On the other hand, by Lemma \ref{lem eigenvalue problem} and \eqref{eq:L Lambda}, we obtain that for each $i=1,\;2,\;\dots,\;n$,
\begin{equation*}
    \mathcal{L} w_i + \frac{1}{2} w_i=0,\quad\text{in }B_{R}(0),
\end{equation*}
and
\begin{equation*}
    \mathcal{L} \Lambda(w)_{i}+ \frac{1}{2} \Lambda(w)_i+\Lambda(w)_i-e^{w}w_{i}\Lambda=0,\quad\text{in }B_{R}(0).
\end{equation*}
Notice that the first order term in $\mathcal{L}$ grows at most linearly. By applying standard elliptic regularity (see \cite[Theorem 9.20]{Gilbarg-T2001}) on  $B_{1/R}(x)$ (for any $x\in B_{R-2}(0)$), we get
\begin{equation}\label{eq:nabla w pointwise estimate}
|\nabla w(x)|^2
\le
C R^n \int_{B_{1/R}(x)} |\nabla w|^2 {\rm d}y .
\end{equation}
\begin{equation}\label{eq:nabla Lambda pointwise estimate}
|\nabla \Lambda(w)(x)|^2
\le
C R^n \int_{B_{1/R}(x)} |\nabla \Lambda(w)|^2 {\rm d}y
+ C R^{2n-2} e^{-\frac{R}{4}}.
\end{equation}
By \eqref{eq:nabla w pointwise estimate}, \eqref{eq:nabla Lambda pointwise estimate} and \eqref{eq:integral decay}, we derive \eqref{eq:nabla w and nabla Lambda estimate}.
\end{proof}

Now we are ready to prove Proposition \ref{prop 1}.

\begin{proof}[Proof of Proposition \ref{prop 1}]
Taking $R_{0}=5\sqrt{n}$ and $\varepsilon=\frac{\varepsilon_{0}}{100}$ in Lemma \ref{lem large scale control gives local rigidity}, then we obtain
\begin{equation}\label{eq R=5sqrtn}
\mathrm{dist}_{5\sqrt{n}}(w,0) + \mathrm{dist}_{5\sqrt{n}}\left(\Lambda(w), 2\right)
\le \frac{\varepsilon_{0}}{100}.
\end{equation}
For any $x\in B_{R-3}(0)$ with $|x_{\ast}|>5\sqrt{n}$, by Lemma \ref{lem pointwise nabla w and Lambda estimate}, we have
\begin{equation}\label{eq w difference bdd}
\left|w(x) - w\left(\frac{x_{\ast}}{|x_{\ast}|}\right)\right|
\le
\left|\int_1^{|x_{\ast}|} \frac{x_{\ast}}{|x_{\ast}|} \cdot \nabla w\left(\xi \frac{x_{\ast}}{|x_{\ast}|}\right) {\rm d}\xi\right|
\le C R^{n+1} e^{-R/8}.
\end{equation}
The same argument also gives
\begin{equation}\label{eq Lambda difference bdd}
\left|\Lambda(w)(x) - \Lambda (w)\left(\frac{x_{\ast}}{|x_{\ast}|}\right)\right|
\le C R^{n+1} e^{-R/8}.
\end{equation}
Now combining \eqref{eq R=5sqrtn} with \eqref{eq w difference bdd}, we get
\begin{equation*}
|w(x)-0| \le \left|w\left(\frac{x_\ast}{|x_{\ast}|}\right)\right| + \left|w(x)-w\left(\frac{x_{\ast}}{|x_{\ast}|}\right)\right| \le \frac{\varepsilon_0}{100} + CR^{n+1}e^{-R/8}.
\end{equation*}
Hence, for $R$ large enough,
\begin{equation*}
|w(x)| \le \frac{\varepsilon_0}{20} \quad \text{for all } x \in B_{R-3}(0).
\end{equation*}
Likewise, using  \eqref{eq R=5sqrtn} together with \eqref{eq Lambda difference bdd}, we derive
\begin{equation*}
|\Lambda(w)(x) - 2| \le \left|\Lambda\left(\frac{x_{\ast}}{|x_{\ast}|}\right) - 2\right| +\left|\Lambda(w)(x)-\Lambda(w)\left(\frac{x_{\ast}}{|x_{\ast}|}\right)\right|  \le \frac{\varepsilon_0}{100} + CR^{n+1}e^{-R/8}.
\end{equation*}
Taking $R$ large enough, we have
\begin{equation*}
|\Lambda (w)(x)-2| \le \frac{\varepsilon_0}{20} \quad \text{for all } x \in B_{R-3}(0).
\end{equation*}
Therefore we conclude that
\begin{equation*}
\mathrm{dist}_{R-3}(w,0) + \mathrm{dist}_{R-3}\left(\Lambda(w), 2\right)
\le \frac{\varepsilon_{0}}{10}.
\end{equation*}
\end{proof}

Next, we prove Proposition \ref{prop 2}.
\begin{proof}[Proof of Proposition \ref{prop 2}]
Set
\begin{equation}\label{eq:u}
    u(x,t)=-\log(-t)+w\left(\frac{x}{\sqrt{-t}}\right).
\end{equation}
By the boundedness of $w$, we know that there exists a small constant $s_{\ast}$ such that
\begin{equation*}
    |u| \le C, \quad \text{in } B_{R-4}(0) \times \left[-1 - \frac{s_\ast}{2}, -1 + s_\ast\right].
\end{equation*}
Then by standard parabolic regularity theory, there exists a universal constant $C$ such that
\begin{equation}\label{eq difference for u and ut}
    |\partial_t u| + |\partial_{tt} u| \le C,\quad \text{in } B_{R-5}(0) \times [-1, -1 + s_\ast].
\end{equation}
Integrating in $t$ yields that for any $x\in B_{R-5}(0)$,
\begin{equation*}
    | u(x, -1)-u(x, -1+s_\ast) | + |\partial_t u(x, -1)-\partial_t u(x, -1+s_\ast) | \leq C s_\ast.
\end{equation*}

By \eqref{eq:u}, we derive
\begin{align*}
\partial_t u(x,t)
&= \frac{1}{-t}
+ \frac{1}{2(-t)}\frac{x}{\sqrt{-t}}\cdot\nabla w\!\left(\frac{x}{\sqrt{-t}}\right) \\
&= \frac{1}{2(-t)}
\left[
2+\frac{x}{\sqrt{-t}}\cdot\nabla w\!\left(\frac{x}{\sqrt{-t}}\right)
\right] \\
&= \frac{1}{2(-t)}\Lambda(w)\!\left(\frac{x}{\sqrt{-t}}\right).
\end{align*}
In particular,
\begin{equation}\label{eq:u and ut at x,-1}
u(x,-1)=w(x) \quad \text{and} \quad \partial_t u(x,-1)=\frac{1}{2}\Lambda (w)(x).
\end{equation}
Inserting \eqref{eq:u and ut at x,-1} into \eqref{eq difference for u and ut}, we deduce that for any $x\in B_{R-5}(0)$
\begin{equation}\label{eq difference for s*}
    \left|w(x)-\left[w\left(\frac{x}{\sqrt{1-s_{\ast}}}\right)-\log(1-s_{\ast})\right]\right|+\frac{1}{2}\left|\Lambda(w)(x)-\frac{1}{1-s_{\ast}}\Lambda(w)\left(\frac{x}{\sqrt{1-s_{*}}}\right)\right|\leq Cs_{\ast}.
\end{equation}
A combination of \eqref{eq condition in prop2} and \eqref{eq difference for s*} leads to
\begin{equation*}
\left| w\left(\frac{x}{\sqrt{1-s_*}}\right)-\log(1-s_{\ast}) \right|
+
\left| \frac{1}{1-s_{\ast}}\Lambda(w)\left(\frac{x}{\sqrt{1-s_*}}\right) - 2 \right|
\le \frac{\varepsilon_0}{10} + Cs_{\ast} .
\end{equation*}
First, this implies
\begin{equation*}
\left| \Lambda(w)\left(\frac{x}{\sqrt{1-s_*}}\right)  \right|
\le \frac{\varepsilon_0}{10} +2+ Cs_{\ast}.
\end{equation*}
Then a Taylor expansion of $\log(1-s_{\ast})$ and $\frac{1}{1-s\ast}$ yields
\begin{equation}\label{Taylor expansion}
\left| w\left(\frac{x}{\sqrt{1-s_*}}\right) \right|
+ \left| \Lambda(w)\!\left(\frac{x}{\sqrt{1-s_*}}\right) - 2 \right|
\le \frac{\varepsilon_0}{10}+ C s_* .
\end{equation}
We now select  $s_{\ast}$ small enough such that
\begin{equation*}
    \frac{\varepsilon_0}{10}+ C s_*\leq \varepsilon_{0},
\end{equation*}
and then select $R_{2}$ large enough such that, for some $\theta>0$,
\begin{equation*}
\frac{R_2-5}{\sqrt{1-s_*}} \ge (1+\theta)R_2.
\end{equation*}
Therefore, \eqref{Taylor expansion} implies that
\begin{equation*}
\operatorname{dist}_{(1+\theta)R}\bigl(w,0\bigr)
+
\operatorname{dist}_{(1+\theta)R}\left(\Lambda(w),2\right)
\le
\varepsilon_{0}.
\end{equation*}
\end{proof}

We now apply Propositions \ref{prop 1} and \ref{prop 2} to prove the rigidity result.
\begin{proof}[Proof of Theorem \ref{theorem 2}]
Take $\widetilde{R}:=2\max\{R_{1},R_{2}\}$, where $R_{1}$ and $R_{2}$ are the constants in Propositions \ref{prop 1} and \ref{prop 2} respectively.

Since $w$ is a solution of \eqref{stationary problem} and satisfies \eqref{eq:rigidity}, standard interior elliptic estimate implies that if $\widetilde{\varepsilon}$ is small enough, then $w$ is smooth in $B_{\widetilde{R}(0)}$ and it satisfies
\begin{equation*}
\mathrm{dist}_{\widetilde{R}-1}(w,0) + \mathrm{dist}_{\widetilde{R}-1}\left(\Lambda(w), 2\right)
\le \frac{\varepsilon_{0}}{10}.
\end{equation*}
We proceed to iterate Propositions \ref{prop 1} and \ref{prop 2} alternately. Since $\widetilde{R}$ depends only on $n$ and $m$, $\widetilde{\varepsilon}$ can be selected as a universal positive constant. Note that Proposition \ref{prop 2} increases the scale by a factor greater than one, while Proposition \ref{prop 1} only incurs a fixed loss in radius in order to obtain the improved estimate. Therefore, as long as the starting scale $\widetilde{R}$ is sufficiently large, the argument may be iterated, and we arrive at
\begin{equation}\label{eq: sup rigidity}
\sup_{y \in \mathbb{R}^n} |w(y)|
+ \sup_{y \in \mathbb{R}^n} |\Lambda(w)(y) - 2|
\le \frac{\varepsilon_0}{10}.
\end{equation}
\eqref{eq: sup rigidity} implies that $w$ is a solution of \eqref{stationary problem} such that $\Lambda(w)$ does not change sign. By Lemma \ref{lem not change sign}, we conclude that $w\equiv0$.
\end{proof}
\section*{Acknowledgments}
Yuxia Guo was supported by the National Natural Science Foundation of China(No. 12271283) and  National Key R\&D Program (2023YFA1010002).

\section*{Statements and Declarations}

The authors confirm that there are no relevant financial or non-financial competing interests to report.
	
\section*{Data Availability Statements}

All data generated or analyzed during this study are included in this article.

\end{document}